\documentclass[11pt]{amsart}
\usepackage{amsmath}
\usepackage{amsthm}
\usepackage{amssymb}
\usepackage{verbatim}
\usepackage{color}
\newtheorem{theorem}{Theorem}    %[section]

\newtheorem{lemma}[theorem]{Lemma}

\newtheorem{remark}[theorem]{Remark}
\newtheorem{definition}[theorem]{Definition}
\theoremstyle{definition}

\numberwithin{theorem}{section} \numberwithin{theorem}{section}
\numberwithin{equation}{section}

\def\Rn{{\mathbb{R}^n}}

\def\supp{\operatorname{supp}}

\def\esss{\operatornamewithlimits{ess\,sup}}

\allowdisplaybreaks

\begin{document}
\title[Variational characterizations of Hardy spaces and $BMO$ spaces]
{Variational characterizations of weighted Hardy spaces and weighted $BMO$ spaces}
%title of paper and the running head option

\author{Weichao Guo, Yongming Wen, Huoxiong Wu$^*$ and Dongyong Yang}
%%%%%%%%%%%%%%% footnote %%%%%%%%%%%%%%%%

%In case \subjclass[2000] command is not effective
%(or the version of amsart.cls is old), write as follows instead:
%\renewcommand{\thefootnote}{\fnsymbol{footnote}}
%\footnote[0]{2000\textit{ Mathematics Subject Classification}.

%Primary 00; Secondary 00.}
%
\thanks{$^*$Corresponding author.}
\thanks{Supported by the NNSF of China (Nos. 11771358, 11871101, 11971402, 11871254, 11701112).}%, 11671414).}
%%%%%%%%%%%% Authors addresses %%%%%%%%%%%%%
\address{School of Science, Jimei University, Xiamen 361021, China} \email{weichaoguomath@gmail.com}
\address{School of Mathematics and Statistics, Minnan Normal University, Zhangzhou 363000,  China} \email{wenyongmingxmu@163.com}
\address{School of Mathematical Sciences, Xiamen University, Xiamen 361005, China} \email{huoxwu@xmu.edu.cn}
\address{School of Mathematical Sciences, Xiamen University, Xiamen 361005, China} \email{dyyang@xmu.edu.cn}

%\address{}
%\email{}

%%%%%%%%%%%%%%%%%%%%%%%%%%%%%%%%%%%%%%%%%

\begin{abstract}
This paper obtains new characterizations of weighted Hardy spaces and certain weighted $BMO$ type spaces via the boundedness of variation operators associated with approximate identities and their commutators, respectively.
\end{abstract}

\maketitle

\section{Introduction and main results}
Variational inequalities originated from the intention of improving the well-known Doob maximal inequality. Relied upon the work of L\'{e}pingle \cite{Le}, Bourgain \cite{Bo} obtained the corresponding variational estimates for the Birkhoff ergodic averages and pointwise convergence results. This work has set up a new research subject in harmonic analysis and ergodic theory. Afterwards, the study of variational inequalities has been spilled over into harmonic analysis, probability and ergodic theory. Particularly, the classical work of $\rho$-variation operators for singular integrals was given in \cite{CaJRW1}, in which the authors obtained the $L^p$-bounds and weak type (1,1) bounds for $\rho$-variation operators of truncated Hilbert transform if $\rho>2$, and then extended to higher dimensional cases in \cite{CaJRW2}. For further studies, we refer readers to \cite{CheDHL,DHL,GT,MaTX1,MaTX2,WWZ} etc. for variation operators of singular integrals with rough kernels and weighted cases, \cite{BetFHR,CheDHL2,HLP,LW,WGW,WWZ} etc. for variation operators of commutators.

Here, we will focus on the variation operators associated with approximate identities. For the special case, the variation operators associated with heat and Poisson semigroups, Jones et al. \cite{JR} and Crescimbeni et al. \cite{CrMacMTV} independently established the $L^p$-bounds and weak type $(1,1)$ bounds in the different approach. Recently, Liu \cite{Liu} generalized the results in \cite{CrMacMTV, JR} to the variation operators associated to approximate identities and obtained a variational characterization of Hardy spaces. In this paper, one of our main purposes is to extend the results in \cite{Liu} to the weighted cases, and give a new characterization of weighted Hardy spaces via variation inequalities associated with approximate identities. Meanwhile, we will also consider the weighted variation inequalities associated with commutators of approximate identities and aim to provide new characterizations of certain weighted $BMO$ spaces. Before stating our results, we first recall some relevant notation and definitions.

Given a family of complex numbers $\mathfrak{a}:=\{a_t\}_{t\in I}$ with $I\subset (0, +\infty)$. For $\rho>1$, the $\rho$-variation of $\mathfrak{a}$ is given by
\begin{align*}
\|\mathfrak{a}\|_{\mathcal{V}_\rho}:=\sup\Big(\sum_{k\geq1}|a_{t_k}-a_{t_{k+1}}|^\rho
\Big)^{1/\rho},
\end{align*}
where the supremum is taken over all finite decreasing sequences $\{t_k\}$ in $I$. From the definition of $\rho$-variation,  it is easy to check that
\begin{equation}\label{eq 1.2}
\sup_{t\in I}|a_t|\leq|a_{t_0}|+\|\mathfrak{a}\|_{\mathcal{V}_\rho}
\end{equation}
holds for arbitrary $t_0\in I$.

Let $\mathcal{F}:=\{F_t\}_{t>0}$ be a family of operators. Then the $\rho$-variation of the family $\mathcal{F}$ is defined as
\begin{align*}
\mathcal{V}_\rho(\mathcal{F}f)(x):=\|\{F_tf(x)\}_{t>0}\|_{\mathcal{V}_\rho}.
\end{align*}
In particular, let $\phi\in\mathcal{S}(\Rn)$ satisfying $\int_{\Rn}\phi(x)dx=1$, where $\mathcal{S}(\Rn)$ is the space of Schwartz functions. We consider the following family of operators
\begin{equation}\label{eq 1.3}
\Phi\star f(x):=\{\phi_t\ast f(x)\}_{t>0},
\end{equation}
where $\phi_t(x):=t^{-n}\phi(x/t)$.
Then the $\rho$-variation of families $\Phi\star f$ is defined by
\begin{align}\label{simple1}
\mathcal{V}_\rho(\Phi\star f)(x)=\sup_{\{t_k\}\downarrow0}\Big(\sum_{k\geq1}
|\phi_{t_{k}}\ast f(x)-\phi_{t_{k+1}}\ast f(x)|^\rho\Big)^{1/\rho}.
\end{align}
And for $b\in L^1_{\rm loc}(\mathbb{R}^n)$ and $\phi$ being as above, we define $(\Phi\star f)_b$, the family of commutators associated with approximate identities, by
\begin{equation*}
(\Phi\star f)_b(x):=\{b(x)(\phi_t\ast f)(x)-\phi_t\ast (bf)(x)\}_{t>0},
\end{equation*}
where
$$b(x)(\phi_t\ast f)(x)-\phi_t\ast (bf)(x):=\int_{\Rn}\frac{1}{t^n}\phi\Big(\frac{x-y}{t}\Big)(b(x)-b(y))f(y)dy,$$
and the corresponding $\rho$-variation operator by
\begin{align}\label{simple2}
\mathcal{V}_\rho((\Phi\star f)_b)(x)&:=\sup_{\{t_k\}\downarrow0}\Big(\sum_{k\geq1}
\Big|b(x)(\phi_{t_k}\ast f)(x)-\phi_{t_k}\ast (bf)(x)\\
&\qquad\qquad-b(x)(\phi_{t_{k+1}}\ast f)(x)+\phi_{t_{k+1}}\ast (bf)(x)
\Big|^\rho\Big)^{1/\rho}\nonumber.
\end{align}

We now recall some revelent facts about weighted Hardy spaces; see, for example, \cite{Gar,ST}.
Let $\phi\in\mathcal{S}(\mathbb{R}^n)$ with $\int_{\mathbb{R}^n}\phi(x)dx=1$. For $f\in\mathcal{S}'(\mathbb{R}^n)$, define the maximal function $M_\phi$ by
\begin{align}\label{eq1.5}
M_\phi f(x):=\sup_{t>0}|f\ast\phi_t(x)|.
\end{align}
Then for $\omega\in A_\infty$, the Muckenhoupt class, and $0<p<\infty$, define the weighted Hardy spaces $H^p(\omega)$ by
\begin{align*}
H^p(\omega):=\{f\in\mathcal{S}'(\mathbb{R}^n):M_\phi f\in L^p(\omega)\}
\end{align*}
with the quasi-norm $\|f\|_{H^p(\omega)}:=\|M_\phi f\|_{L^p(\omega)}$. When $\omega\equiv 1$, we denote $H^p(\omega)$ by $H^p(\mathbb{R}^n)$. It is well-known that the space $H^p(\omega)$ is independent of the choice of $\phi$ and $H^p(\omega)=L^p(\omega)$ for $p>1$ and $\omega\in A_p$.

Now we formulate our first main result, which is the weighted version of the corresponding result in \cite{Liu}, as follows.

\begin{theorem}\label{theorem1.1}
Let $\phi\in \mathcal S(\mathbb R^n)$ with $\int_{\Rn}\phi(x)dx=1$, $\rho>2$. Then,

\medskip

{\rm (i)}\, when $n/(n+1)<p\leq1$ and $\omega\in A_{p(n+1)/n}$, for
$f\in\mathcal{S}'(\mathbb{R}^n)$ and any $t>0$, $f\in H^p(\omega)$ if and only if $\phi_t\ast f\in L^p(\omega)$ and $\mathcal{V}_\rho(\Phi\star f)\in L^p(\omega)$, with
\begin{align*}
\|\phi_t\ast f\|_{L^p(\omega)}+\|\mathcal{V}_\rho(\Phi\star f)\|_{L^p(\omega)}
\lesssim \|f\|_{H^p(\omega)}
\lesssim\|\phi_t\ast f\|_{L^p(\omega)}+\|\mathcal{V}_\rho(\Phi\star f)\|_{L^p(\omega)}\nonumber;
\end{align*}

\medskip

{\rm (ii)}\, when $1<p<\infty$ and $\omega\in A_p$,
\begin{align}\label{eq1.6}
\|\mathcal{V}_\rho(\Phi\star f)\|_{L^p(\omega)}\lesssim[\omega]_{A_p}^
{\max\{1,\frac{1}{p-1}\}}\|f\|_{L^p(\omega)},\qquad \forall\, f\in L^p(\omega);
\end{align}
and for any $t>0$, a function $f\in H^p(\omega)$ if and only if $\phi_t\ast f\in L^p(\omega)$ and $\mathcal{V}_\rho(\Phi\star f)\in L^p(\omega)$, with
\begin{align*}
\|\phi_t\ast f\|_{L^p(\omega)}+\|\mathcal{V}_\rho(\Phi\star f)\|_{L^p(\omega)}
\lesssim \|f\|_{H^p(\omega)}\leq \|\phi_t\ast f\|_{L^p(\omega)}+\|\mathcal{V}_\rho(\Phi\star f)\|_{L^p(\omega)}.
\end{align*}

Here the implicit constants are independent of $f$ and $t$.
\end{theorem}

\medskip

On the other hand, recalling the commutators
$$[b,T]f(x):=b(x)Tf(x)-T(bf)(x),$$
where $b$ is a given locally integral function and $T$ is a linear operator, Coifman, Rochberg and Weiss \cite{CRW} showed that $b\in BMO(\mathbb{R}^n)$ if and only if $[b,R_j]$ is bounded on $L^p(\mathbb{R}^n)$ for $1<p<\infty$, where $R_j$ is the $j$-th Riesz transform on $\mathbb{R}^n$, $j=1,\,\cdots,n$. For $n=1$, Bloom \cite{Bl} extended the above result to the following weighted case:
$$\|[b,H]\|_{L^p(\mu)\to L^p(\lambda)}\simeq \|b\|_{ BMO_\nu(\mathbb R)}\, \quad \text{for}\,\, 1<p<\infty,\,\mu,\,\lambda\in A_p,$$
where $H$ is the Hilbert transform, $\nu=(\mu/\lambda)^{1/p}$, and $BMO_\nu(\mathbb R^n)$ is the following weighted $BMO$ space defined by
$${ BMO}_\nu(\mathbb R^n):=\{b\in L^1_{\rm loc}(\mathbb{R}^n):\,\, \|b\|_{BMO_\nu(\mathbb R^n)}:=\sup_{B\subset\mathbb{R}^n}\frac1{\nu(B)}\int_B|b(y)-\langle b\rangle_B|dy<\infty\}.$$
Here the supremum is taken over all balls, $\langle b\rangle_B:=|B|^{-1}\int_B b(y)dy$. Subsequently, a lot of attentions have been paid to this topic. We refer to \cite{HLW,HRS,LOR1} and therein references for recent works.

In addition, it is well-know that when $b\in BMO(\mathbb{R}^n)$, for singular integral $T$, $[b,T]$ may not map $H^1(\mathbb{R}^n)$ boundedly into $L^1(\mathbb{R}^n)$ (see \cite[Remark]{Pa}). To investigate the $H^1-L^1$ bound of $[b, T]$, Liang, Ky and Yang \cite{LKY} introduced the following weighted BMO type space
$$\mathcal{BMO}_\omega(\mathbb{R}^n):=\{b\in L^1_{\rm loc}(\mathbb{R}^n):\, \|b\|_{\mathcal{BMO}_\omega(\mathbb{R}^n)}<\infty\},$$
where
$$\|b\|_{\mathcal{BMO}_\omega(\mathbb{R}^n)}:=\sup_{B\subset\mathbb{R}^n}\frac 1{\omega(B)}\int_{B^c}\frac{\omega(x)}{|x-x_B|^n}dx\int_B|b(y)-\langle b\rangle_B|dy$$
for $\omega\in A_\infty$ and $\int_{\mathbb{R}^n}\frac{\omega(x)}{1+|x|^n}dx<\infty$, the supremum is taken over all balls $B:=B(x_B,r)$. It is clear that $\mathcal{BMO}_\omega(\mathbb{R}^n)\subsetneq BMO(\mathbb{R}^n)$ (see \cite{LKY}). And in \cite{LKY}, the authors showed that for $\delta$-Calder\'on-Zygmund operator $T$ and $\omega\in A_{(n+\delta)/n}$, $b\in \mathcal{BMO}_\omega(\mathbb{R}^n)$ if and only if $[b,T]$ is bounded from $H^1(\omega)$ to $L^1(\omega)$.

Inspired by the above results, it is natural to ask whether the corresponding characterizations can be given via variation operators of commutators associated with approximate identities. Our next theorems will give a confirmative answer to this question.

\begin{theorem}\label{new theorem1}
Let $\phi\in \mathcal S(\mathbb R^n)$ with $\int_{\Rn}\phi(x)dx=1$, $1<p<\infty$, $\mu,\lambda\in A_{p}$ and $\nu=(\mu\lambda^{-1})^{1/p}$. Then for $\rho>2$, the following statements are equivalent:
\begin{itemize}
\item [(1)]$\|\mathcal{V}_\rho((\Phi\star f)_b)\|_{L^p(\lambda)}\lesssim
([\mu]_{A_p}[\lambda]_{A_p})^{\max\{1,\frac{1}{p-1}\}}\|f\|_{L^p(\mu)}$ ;
\item [(2)]$b\in BMO_\nu(\mathbb{R}^n)$.
  \end{itemize}
\end{theorem}

\begin{theorem}\label{new theorem} Let $\phi\in \mathcal S(\mathbb R^n)$ with $\int_{\Rn}\phi(x)dx=1$, $\rho>2$.
Assume that $b\in BMO(\mathbb{R}^n)$, $\omega\in A_{1}$ with $\int_{\mathbb{R}^n}\frac{\omega(x)}{1+|x|^n}dx<\infty$. Then the following statements are equivalent:
\begin{itemize}
\item [(1)]$\mathcal{V}_\rho((\Phi\star f)_b)$ is bounded from $H^1(\omega)$ to $L^{1}(\omega)$.
\item [(2)]$b\in \mathcal{BMO}_\omega(\mathbb{R}^n)$.
  \end{itemize}
\end{theorem}

\begin{remark}\label{r-heat and poisn}
Consider the  heat semigroup $\mathcal{W}:=\{e^{t\Delta}\}_{t>0}$ and the Poisson semigroup $\mathcal{P}:=\{e^{-t\sqrt{-\Delta}}\}_{t>0}$
associated to $\Delta=\sum_{i=1}^{n}\frac{\partial^2}{\partial x_{i}^2}$.
Since the heat kernels $W_t(x):=(\pi t)^{-n/2}e^{-|x|^2/t}$  belongs to $\mathcal{S}(\Rn)$ and satisfies $\int_{\Rn}W_t(x)dx=1$, so Theorems \ref{theorem1.1}-\ref{new theorem} hold for the variation operators associated with $\mathcal{W}$ and their commutators. Similarly, the same conclusions are true for the the variation operators associated with $\mathcal{P}$ and their commutators.
\end{remark}

The rest of the paper is organized as follows. After providing the weighted estimates of variation operators in Section 2, we will prove Theorem \ref{theorem1.1} in Section 3. The proofs of Theorems \ref{new theorem1} and \ref{new theorem} will be given in Sections 4 and 5, respectively.

We end this section by making some conventions. Denote $f\lesssim g$, $f\thicksim g$ if $f\leq Cg$ and $f\lesssim g \lesssim f$, respectively.  For any ball $B:=B(x_0,r)\subset \mathbb{R}^n$, $\langle f\rangle_B$ means the mean value of $f$ over $B$, $\chi_B$ represents the characteristic function of $B$, $\int_B\omega(y)dy$ is denoted by $\omega(B)$. For $a\in\mathbb{R}$, $\lfloor a\rfloor$ is the largest integer no more than $a$.

\section{The weighted estimate of variation operators}
In this section, we establish the weighted estimate of variation operators associated with approximations to the identity, which is useful
in the proof of Theorem \ref{theorem1.1}. We begin with the following definition of $A_p$ weights.

A weight $\omega$ is a non-negative locally integrable function on $\mathbb{R}^n$. Let $1<p<\infty$.
We say that $\omega\in A_p$ if there exists a positive constant $C$ such that
\begin{align*}
[\omega]_{A_p}:=\sup_{Q}\Big(\frac{1}{|Q|}\int_{Q}\omega(y)dy\Big)\Big(\frac{1}{|Q|}
\int_{Q}\omega(y)^{1-p'}dy\Big)^{p-1}\leq C,
\end{align*}
where $1/p+1/p'=1$ and the supremum is taken over all cubes $Q\subset \mathbb{R}^n$. When $p=1$, we say that $\omega\in A_1$ if
\begin{align*}
[\omega]_{A_1}:=\Big\|\frac{M\omega}{\omega}\Big\|_{L^\infty(\Rn)}<\infty,
\end{align*}
where $M$ is the Hardy-Littlewood maximal operator.
For $\omega\in A_\infty:=\cup_{1\leq p<\infty}A_p$, the $A_\infty$ constant is given by
$$[\omega]_{A_\infty}:=\sup_{Q\subset \Rn}\frac{1}{\omega(Q)}\int_QM(\chi_Q\omega)(x)dx.$$

Now we recall the definitions of dyadic lattice, sparse family and sparse operator; see, for example,  \cite{Ler,LOR,Perey}. Given a cube $Q\subset\mathbb{R}^n$, let $\mathcal{D}(Q)$ be the set of cubes obtained by repeatedly subdividing $Q$ and its descendants into $2^n$ congruent subcubes.
\begin{definition}
A collection of cubes $\mathcal{D}$ is called a dyadic lattice if it satisfies the following properties:\\
$(1)$ if $Q\in\mathcal{D}$, then every child of $Q$ is also in $\mathcal{D}$;\\
$(2)$ for every two cubes $Q_1, Q_2\in\mathcal{D}$, there is a common ancestor $Q\in\mathcal{D}$ such that $Q_1, Q_2\in\mathcal{D}(Q)$;\\
$(3)$ for any compact set $K\subset\mathbb{R}^n$, there is a cube $Q\in\mathcal{D}$ such that $K\subset Q$.
\end{definition}

\begin{definition}
A subset $\mathcal{S}\subset\mathcal{D}$ is called an $\eta$-sparse family with $\eta\in(0,1)$ if for every cube $Q\in\mathcal{S}$, there is a measurable subset $E_Q\subset Q$ such that $\eta|Q|\leq|E_Q|$, and the sets $\{E_Q\}_{Q\in\mathcal{S}}$ are mutually disjoint.
\end{definition}

Let $\mathcal S$ be a sparse family. Define the sparse operator $\mathcal{T}_{\mathcal{S}}$ by
\begin{align}\label{e-spar fam defn}
\mathcal{T}_{\mathcal S}f(x):=\sum_{Q\in \mathcal S}\langle|f|\rangle_Q\chi_Q(x).
\end{align}
In \cite{Ler,LOR}, the authors  proved the following strong and weak type bounds for sparse operators: for any operator $\mathcal{T}_{\mathcal{S}}$ as in \eqref{e-spar fam defn} and $\omega\in A_p$ with $1<p<\infty$,
\begin{equation}\label{eq 2.1}
\|\mathcal{T}_{\mathcal{S}}f\|_{L^p(\omega)}\lesssim[\omega]_{A_p}^
{\max\{1,\frac{1}{p-1}\}}\|f\|_{L^p(\omega)},
\end{equation}
and  for $\omega\in A_1$ and $\alpha>0$,
\begin{equation}\label{eq 2.3}
\alpha\omega(\{x\in\mathbb{ R}^n:\mathcal{T}_{\mathcal S}f(x)>\alpha\})\lesssim[\omega]_{A_1}
\log(e+[\omega]_{A_\infty})\|f\|_{L^1(\omega)}.
\end{equation}

We now give the following pointwise estimate of variation operators in terms of sparse operators.
\begin{lemma}\label{prop2.5}
Let $\rho>2$. Then for each $f\in L_c^\infty(\mathbb{R}^n)$, there exist $3^n$ dyadic lattices $\mathcal{D}^j$ and sparse families $\mathcal{S}_j\subset\mathcal{D}^j$ such that for a.e. $x\in\mathbb{R}^n$,
$$\mathcal{V}_\rho(\Phi\star f)(x)\lesssim\sum_{j=1}^{3^n}
\mathcal{T}_{\mathcal{S}_j}(f)(x).$$
\end{lemma}

To prove Lemma \ref{prop2.5}, we introduce the maximal function $\mathcal{M}_{\mathcal{V}_\rho(\Phi)}$ by
\begin{align*}
\mathcal{M}_{\mathcal{V}_\rho(\Phi)}f(x):=\sup_{Q\ni x}\esss_{\xi\in Q}\mathcal{V}_\rho(\Phi\star (f\chi_{\Rn\backslash3Q}))(\xi).
\end{align*}
Since $\mathcal{V}_\rho(\Phi\star f)$ is of weak type (1, 1) (see \cite[Theorem 2.6]{Liu}), then arguing as in \cite[Theorem 3.1]{Ler} (see also \cite[Remark 4.3]{Ler}), Lemma \ref{prop2.5} is an immediate consequence of that $\mathcal{M}_{\mathcal{V}_\rho(\Phi)}$ is of weak $(1,1)$. Therefore, we need only to check the following result.

\begin{lemma}\label{lm2.4}
For $\rho>2$, $\mathcal{M}_{\mathcal{V}_\rho(\Phi)}$ is bounded from $L^1(\mathbb{R}^n)$ to $L^{1,\infty}(\mathbb{R}^n)$.
\end{lemma}

\begin{proof}
For any $f\in L^1(\mathbb R^n)$ and $x\in \mathbb R^n$, we first show that for any cube $Q\ni x$,
\begin{align}\label{e-var ptw esti}
\mathcal{V}_\rho(\Phi\star (f\chi_{\mathbb{R}^n\backslash3Q}))(\xi)\lesssim  Mf(x)+M_{1/2}(\mathcal{V}_\rho(\Phi\star f))(x)\quad \text{for a.e.}\,\, \xi\in Q,
\end{align}
where $M_{1/2}(f):=(M(|f|^{1/2}))^2$, and
the implicit constant is independent of $f$, $x$ and $\xi$.

Indeed, for any cube $Q\ni x$, a.e. $\xi\in Q$ and $z\in Q$, we write
\begin{align*}
\mathcal{V}_\rho(\Phi\star (f\chi_{\mathbb{R}^n\backslash3Q}))(\xi)
&\le\sup_{\{t_k\}\downarrow0}\Big(\sum_k\Big|\int_{\mathbb{R}^n\backslash3Q}
[(\phi_{t_k}(\xi-y)-\phi_{t_{k+1}}(\xi-y))\\
&\qquad-(\phi_{t_k}(z-y)-\phi_{t_{k+1}}(z-y))]f(y)dy\Big|
^\rho\Big)^{1/\rho}\\
&\qquad+\mathcal{V}_\rho(\Phi\star (f\chi_{3Q}))(z)+\mathcal{V}_\rho(\Phi\star f)(z)\\
&=:J(\xi,z)+\mathcal{V}_\rho(\Phi\star (f\chi_{3Q}))(z)+\mathcal{V}_\rho(\Phi\star f)(z).
\end{align*}
By the Minkowski inequality and mean value theorem, we see that for given $\xi,z\in Q$ and $y\in\Rn\backslash3Q$,
\begin{align}\label{eq2.5}
\|\{\phi_t(\xi-y)-\phi_t(z-y)\}_{t>0}\|_{\mathcal{V}_\rho}
&\leq \int_{0}^\infty\Big|\frac{\partial}{\partial t}(\phi_t(\xi-y)-\phi_t(z-y))\Big|dt\\
&\lesssim |z-\xi|\int_{0}^{\infty}t^{-n-2}(1+|\xi-y|/t)^{-n-2}dt\nonumber\\
\lesssim \frac{|z-\xi|}{|\xi-y|^{n+1}}.\nonumber
\end{align}
Consequently,
\begin{align*}
J(\xi,z)&\leq\int_{\mathbb{R}^n\backslash3Q}
\|\{\phi_t(\xi-y)-\phi_t(z-y)\}_{t>0}\|_{\mathcal{V}_\rho}|f(y)|dy\\
&\lesssim\int_{\mathbb{R}^n\backslash3Q}\frac{|z-\xi|}{|z-y|^{n+1}}|f(y)|dy\lesssim Mf(x).
\end{align*}
Then
\begin{align*}
\mathcal{V}_\rho(\Phi\star (f\chi_{\mathbb{R}^n\backslash3Q}))(\xi)
&\lesssim Mf(x)+ \inf_{z\in Q}\left[\mathcal{V}_\rho(\Phi\star (f\chi_{3Q}))(z)+\mathcal{V}_\rho(\Phi\star f)(z)\right]\\
&\lesssim Mf(x)+\Big[\frac{1}{|Q|}\int_Q\mathcal{V}_\rho(\Phi\star (f\chi_{3Q}))(z)^{1/2}dz\Big]^{2}\\
&\qquad+ \Big[\frac{1}{|Q|}\int_Q\mathcal{V}_\rho(\Phi\star f)(z)^{1/2}dz\Big]^{2}\\
&\lesssim Mf(x)+\frac{1}{|Q|}\int_{3Q}|f(y)|dy +M_{1/2}(\mathcal{V}_\rho(\Phi\star f))(x)\\
&\lesssim Mf(x)+M_{1/2}(\mathcal{V}_\rho(\Phi\star f))(x),
\end{align*}
where in the last-to-second inequality, we used the Kolmogorov inequality and the weak type $(1,1)$ of $\mathcal{V}_\rho(\Phi\star f)$~(see \cite{Liu}). Therefore, \eqref{e-var ptw esti} holds.

Recall that $M_{1/2}(f)$ is bounded on $L^{1,\infty}(\mathbb R^n)$ (see \cite{NTV}). Then for $f\in L^1(\mathbb R^n)$, we have $\mathcal{V}_\rho(\Phi\star f)\in L^{1,\infty}(\mathbb R^n)$ and
\begin{align*}
\|M_{1/2}(\mathcal{V}_\rho(\Phi\star f))\|_{L^{1,\infty}(\mathbb R^n)}
\lesssim\|\mathcal{V}_\rho(\Phi\star f)\|_{L^{1,\infty}(\mathbb{R}^n)}
\lesssim\|\mathcal{V}_\rho(\Phi)\|_{L^1(\mathbb{R}^n)\rightarrow L^{1,\infty}(\mathbb{R}^n)}\|f\|_{L^1(\mathbb{R}^n)}.
\end{align*}
This, together with \eqref{e-var ptw esti} and the weak type $(1,1)$ of $M$, implies Lemma \ref{lm2.4}.
\end{proof}

By Lemma \ref{prop2.5}, \eqref{eq 2.1} and \eqref{eq 2.3}, we now obtain the following weighted estimate of variation operators
associated with approximations to the identity.

\begin{theorem}\label{th2.1}
Let $\rho>2$, $1<p<\infty$  and $\omega\in A_p$. Then for any $f\in L^p(\omega)$,
\begin{align*}
\|\mathcal{V}_\rho(\Phi\star f)\|_{L^p(\omega)}\lesssim[\omega]_{A_p}^
{\max\{1,\frac{1}{p-1}\}}\|f\|_{L^p(\omega)}.
\end{align*}
If $\omega\in A_1$, then for any $f\in L^1(\omega)$,
\begin{align*}
\|\mathcal{V}_\rho(\Phi\star f)\|_{L^{1,\infty}(\omega)}\lesssim[\omega]_{A_1}
\log(e+[\omega]_{A_\infty})\|f\|_{L^1(\omega)}.
\end{align*}
\end{theorem}

\begin{remark}\label{new remark}

{\rm (i)}\, Theorem \ref{th2.1} is the quantitative weighted version of Theorem 2.6 in \cite{Liu};

{\rm (ii)} Consider the heat semigroup $\mathcal{W}:=\{e^{t\Delta}\}_{t>0}$ and the Poisson semigroup $\mathcal{P}:=\{e^{-t\sqrt{-\Delta}}\}_{t>0}$.
We remark that Theorem \ref{th2.1} holds for the variation operators associated to $\mathcal{W}$ and $\mathcal{P}$, which can be regarded as the quantitative weighted version of the corresponding results in \cite{CrMacMTV,JR}.
\end{remark}

\section{Characterization of weighted Hardy spaces}
In this section, we give the proof of Theorem \ref{theorem1.1}. Let us begin with recalling some revelent definition and lemma.
\begin{definition}\label{def3.1}
Let $\omega\in A_\infty$, $q_\omega=\inf\{q\in [1,\infty):\,\omega\in A_q\}$, and $0<p\leq1$. Then for $q\in(q_\omega,\infty]$ and $s\in\mathbb{Z}_+:=\{0, 1,2,\cdots,\}$ with $s\geq\lfloor (q_\omega/p-1)n\rfloor$, a function $a$ on $\mathbb{R}^n$ is called a $(p,q,s)_\omega$-atom if the following conditions hold:\\
$(1)$~$\supp a\subset B(x_0,r)$;\\
$(2)$~$\|a\|_{L^q(\omega)}\leq\omega(B(x_0,r))^{1/q-1/p}$;\\
$(3)$~$\int_{B(x_0,r)}a(x)x^\alpha dx=0$ for every multi-index $\alpha$ with $|\alpha|\leq s$.
\end{definition}

We now recall the following useful lemma on the boundedness criterion of an operator from $H^p(\omega)$ to $L^p(\omega)$ with $p\in(0,1]$,  established by Bownik, Li, Yang and Zhou \cite{BLYZ}.

\begin{lemma}{\rm(cf. \cite{BLYZ})}\label{lm3.2}
Let $\omega\in A_\infty$, $q_\omega=\inf\{q\in [1,\infty):\,\omega\in A_q\}$, $0<p\leq1$, and $s\in\mathbb{Z}_+$ with $s\geq\lfloor (q_\omega/p-1)n\rfloor$. Then there exists a unique bounded sublinear operator $\tilde{T}$ from $H^p(\omega)$ to $L^p(\omega)$, which extends $T$, if one of the followings holds:\\
$(1)$~$q_\omega<q<\infty$ and $T:H_{fin}^{p,q,s}(\omega)\rightarrow L^p(\omega)$ is a sublinear operator such that
\begin{align*}
\sup\{\|Ta\|_{L^p(\omega)}, a ~is ~any~ (p,q,s)_\omega-atom\}<\infty,
\end{align*}
where $H_{fin}^{p,q,s}(\omega)$ is the space of all finite linear combinations of $(p,q,s)_\omega$-atoms.\\
$(2)$~$T$ is a sublinear operator defined on continuous $(p,\infty,s)_\omega$-atoms with the property that
\begin{align*}
\sup\{\|Ta\|_{L^p(\omega)}: a~is~any~continuous~(p,\infty,s)_\omega-atom\}<\infty.
\end{align*}
\end{lemma}

Now, we are in the position to prove Theorem \ref{theorem1.1}.

\begin{proof}[Proof of Theorem \ref{theorem1.1}]
(i).\quad When $n/(n+1)<p\leq1$ and $\omega\in A_{p(n+1)/n}$, we first assume that $f\in H^p(\omega)$. Then by the definition of $H^p(\omega)$, we see that $\phi_t\ast f\in L^p(\omega)$ and $\|\phi_t\ast f\|_{L^p(\omega)}\le \|f\|_{H^p(\omega)}$. We now show that
\begin{equation}\label{e-var bdd Hp}
\|\mathcal{V}_\rho(\Phi\star f)\|_{L^p(\omega)}\lesssim \|f\|_{H^p(\omega)}.
\end{equation}
 Invoking Lemma \ref{lm3.2},  it suffices to verify that for some $q_0\in(q_\omega, p(n+1)/n)$ such that $w\in A_{q_0}$
and any $s\in\mathbb{ Z}_+$ with $s\geq\lfloor (q_\omega/p-1)n\rfloor$, there is a positive constant $C$ such that for any $(p,q_0,s)_\omega$-atom $a$,
\begin{equation}\label{e-var bdd atom}
\|\mathcal{V}_\rho(\Phi\star a)\|_{L^p(\omega)}\leq C.
\end{equation}

We assume that $\supp a\subset B:=B(x_0,r)$ and denote $\tilde B:=4B$.
 Then applying the H\"{o}lder inequality, Theorem \ref{th2.1} and Definition \ref{def3.1}, we have
\begin{align}\label{eq 3.1}
\int_{\tilde B}\mathcal{V}_\rho(\Phi\star a)(x)^p\omega(x)dx&\leq[\omega(\tilde B)]^{1-p/{q_0}}
\|\mathcal{V}_\rho(\Phi\star a)\|_{L^{q_0}(\omega)}^p\\
&\lesssim[\omega(B)]^{1-p/{q_0}}\|a\|_{L^{q_0}(\omega)}^{p}\lesssim1.\nonumber
\end{align}

On the other hand, using \eqref{eq2.5} and the vanishing condition of $a$, we have
\begin{align}\label{eq3.4}
\mathcal{V}_\rho(\Phi\star a)(x)
&=\sup_{\{t_k\}\downarrow0}\Big(\sum_k\Big|\int_{\mathbb{R}^n}[(\phi_{t_k}(x-y)
-\phi_{t_{k+1}}(x-y))\\
&\qquad-(\phi_{t_k}(x-x_0)-\phi_{t_{k+1}}(x-x_0))]a(y)dy\Big|
^\rho\Big)^{1/\rho}\nonumber\\
&\quad\lesssim\int_{B}|a(y)|\|\{\phi_t(x-y)-\phi_t(x-x_0)\}_{t>0}\|_{\mathcal{V}_\rho}dy\nonumber\\
&\quad\lesssim\int_B|a(y)|\frac{|y-x_0|}{|x-x_0|^{n+1}}dy,\quad\forall\,\,x\notin\tilde B.\nonumber
\end{align}
By the definition of $A_{q_0}$ and Definition \ref{def3.1}, it yields that
\begin{align*}
\Big(\int_B|a(y)|dy\Big)^p&\leq\Big(\int_B|a(y)|^{q_0}\omega(y)dy\Big)^{p/{q_0}}
\Big(\int_B\omega(y)^{-{q_0}'/{q_0}}dy\Big)^{p/{q_0}'}\\
&\lesssim\omega(B)^{p/{q_0}-1}\omega(B)^{-p/{q_0}}|B|^{p}=\omega(B)^{-1}|B|^{p}.
\end{align*}
Therefore,
\begin{align*}
\int_{\tilde B^c}\mathcal{V}_\rho(\Phi\star a)(x)^p\omega(x)dx&\lesssim\sum_{j=1}^{\infty}
\int_{2^{j+2}B\backslash2^{j+1}B}\frac{r^p}{|x-x_0|^{(n+1)p}}\omega(x)dx
\Big(\int_B|a(y)|dy\Big)^p\\
&\leq\omega(B)^{-1}|B|^{p}\sum_{j=1}^{\infty}\int_{2^{j+2}B}\frac{r^p}{(2^{j+1}r)^{(n+1)p}}\omega(x)dx
\\
&\lesssim\omega(B)^{-1}|B|^{p}\sum_{j=1}^{\infty}2^{-j(pn+p)}|B|^{-p}\omega(2^{j+2}B)
\lesssim 1.
\end{align*}
This, together with the estimate \eqref{eq 3.1}, implies \eqref{e-var bdd atom} and completes the proof of \eqref{e-var bdd Hp}.

Conversely, if $\phi_t\ast f\in L^p(\omega)$ and $\mathcal{V}_\rho(\Phi\star f)\in L^p(\omega)$, then by \eqref{eq 1.2}, $f\in H^p(\omega)$ and
$$\|f\|^p_{H^p(\omega)}=\|M_\phi f\|^p_{L^p(\omega)}\le\|\phi_t\ast f\|^p_{L^p(\omega)}+\|\mathcal{V}_\rho(\Phi\star f)\|^p_{L^p(\omega)}.$$
This finishes the proof of (i).

(ii).\quad When $1<p<\infty$ and $\omega\in A_p$, it follows from Theorem \ref{th2.1} that
\begin{align*}
\|\mathcal{V}_\rho(\Phi\star f)\|_{L^p(\omega)}\lesssim[\omega]_{A_p}^
{\max\{1,\frac{1}{p-1}\}}\|f\|_{L^p(\omega)},\qquad\forall\,\, f\in L^p(\omega).
\end{align*}
That is, (\ref{eq1.6}) holds. Moreover, if $\phi_t\ast f\in L^p(\omega)$ and $\mathcal{V}_\rho(\Phi\star f)\in L^p(\omega)$, then by \eqref{eq 1.2}, $f\in H^p(\omega)$ and
$$\|f\|_{H^p(\omega)}=\|M_\phi f\|_{L^p(\omega)}\le\|\phi_t\ast f\|_{L^p(\omega)}+\|\mathcal{V}_\rho(\Phi\star f)\|_{L^p(\omega)}.$$
In converse, if $f\in H^p(\omega)=L^p(\omega)$, then by \eqref{eq1.6} and the $L^p(\omega)$-boundedness of $M_\phi$,
$$\|\phi_t\ast f\|_{L^p(\omega)}+\|\mathcal{V}_\rho(\Phi\star f)\|_{L^p(\omega)}\lesssim\|M_\phi f\|_{L^p(\omega)}+\|f\|_{L^p(\omega)}
\lesssim\|f\|_{L^p(\omega)}\sim\|f\|_{H^p(\omega)}.$$
This completes the proof of (ii). Theorem \ref{theorem1.1} is proved.
\end{proof}

\section{The characterization of $BMO_\nu(\mathbb{R}^n)$}

This section is concerning with the proof of Theorem \ref{new theorem1}.
We first recall the following relevant notation and the equivalent definition of $BMO_\nu(\mathbb{R}^n)$.

\begin{definition} {\rm (cf. \cite{LOR1})}
By a median value of a real-valued measurable function $f$ over a measure set $E$ of positive finite measure, we mean a possibly non-unique, real number $m_f(E)$ such that
$$\max(|\{x\in E: f(x)>m_f(E)\}|,\,\,|\{x\in E: f(x)<m_f(E)\}|)\leq|E|/2.$$
\end{definition}

In order to introduce the equivalent definition of $BMO_\nu(\mathbb{R}^n)$, we recall the definition of local mean oscillation.
\begin{definition} {\rm (cf. \cite{LOR1})}
For a complex-valued measurable function $f$, we define the local mean oscillation of $f$ over a cube $Q$ by
\begin{equation*}
a_{\tau}(f;Q):=\inf_{c\in \mathbb{C}}((f-c)\chi_Q)^*(\tau|Q|)\hspace{6mm}(0<\tau<1),
\end{equation*}
 where $f^\ast$ denotes the non-increasing rearrangement of $f$.
\end{definition}

For $\tau\in (0,\frac{1}{2^{n+2}}]$,
the following equivalent relation is valid:
\begin{equation}\label{eq4.1}
\|f\|_{BMO_\nu(\mathbb{R}^n)}\sim \sup_Q \frac{|Q|}{\nu(Q)}a_{\tau}(f;Q).
\end{equation}
We refer readers to \cite[Lemma 2.1]{LOR1} for more details.

Now we prove Theorem \ref{new theorem1}.

\begin{proof}[Proof of Theorem \ref{new theorem1}]
We first show that $(1)\Rightarrow(2)$. Without loss of generality, {{we assume that $b$ and $\phi$ are real-valued}},
and $\phi(z)\geq 1$ for $z\in B(z_0,\delta)$, where $|z_0|=1$ and $\delta>0$ is a small constant. For any cube $Q$, denote by
\begin{align*}
P:=Q-10\sqrt{n}\delta^{-1}l_Qz_0
\end{align*}
the cube associated with $Q$. For $0<\tau<1$, by the definition of $a_{\tau}(f;Q)$, there exists a subset $\tilde Q$ of $Q$, such that $|\tilde Q|=\tau|Q|$ and for any $x\in \widetilde Q$,
\begin{align*}
a_{\tau}(b;Q)\leq |b(x)-m_b(P)|.
\end{align*}
Then, by the definition of $m_b(P)$, there exist subsets $E\subset \tilde Q$ and $F\subset P$ such that
\begin{align*}
  |E|=|\tilde Q|/2=\tau|Q|/2,\ |F|=|P|/2=|Q|/2,
\end{align*}
  and
\begin{align*}
  a_{\tau}(b;Q)\leq |b(x)-b(y)|,\quad\forall\, x\in E,\, y\in F,
\end{align*}
and $b(x)-b(y)$ does not change sign in $E\times F$.
Let $$f(x):=\Big(\int_F\mu(x)dx\Big)^{-1/p}\chi_F(x).$$ Then,
\begin{align*}
\mathcal{V}_{\rho}(\Phi\star f)(x)&=
\sup_{t_{k}\downarrow 0}\Big(\sum_{k=1}^{\infty}\Big|\int_{F}(b(x)-b(y))
  (\phi_{t_k}(x-y)-\phi_{t_{k+1}}(x-y))dy\Big|^{\rho}\Big)
  ^{1/\rho}\\
&\qquad\times\Big(\int_F\mu(x)dx\Big)^{-1/p}\\
&\geq\varliminf_{t\rightarrow 0}\Big|\int_{F}(b(x)-b(y))(\phi_{10\sqrt{n}\delta^{-1}l_Q}(x-y)-\phi_{t}(x-y))dy\Big|\\
&\qquad\times\Big(\int_F\mu(x)dx\Big)^{-1/p}.
\end{align*}
For $x\in E\subset Q, y\in F\subset P$, we have
\begin{align*}
x-y\in 2l_QQ_0+10\sqrt{n}\delta^{-1}l_Qz_0\subset (l_QQ_0)^c,\ \ \ \frac{x-y}{10\sqrt{n}\delta^{-1}l_Q}\in \frac{\delta}{5\sqrt{n}}Q_0+z_0\subset B(z_0,\delta),
\end{align*}
where $Q_0$ is the cube centered at origin with side length 1. From this, for $x\in E$, $y\in F$, we have the following estimates
\begin{align*}
  \phi_{10\sqrt{n}\delta^{-1}l_Q}(x-y)
  \gtrsim \frac{1}{|Q|}\phi\Big(\frac{x-y}{10\sqrt{n}\delta^{-1}l_Q}\Big)\geq\frac{1}{|Q|},
\end{align*}
and
\begin{align*}
  \lim_{t\rightarrow 0}\left|\phi_{t}(x-y))\right|
  =
  \lim_{t\rightarrow 0}\frac{1}{t^n}\left|\phi\Big(\frac{x-y}{t}\Big)\right|\lesssim \lim_{t\rightarrow 0}\frac{1}{t^n}\Big(\frac{|x-y|}{t}\Big)^{-n-1}
  \lesssim \lim_{t\rightarrow 0}t(|x-y|)^{-n-1}=0.
\end{align*}
Hence, for $x\in E$,
\begin{align*}
\mathcal{V}_{\rho}(\Phi\star f)(x)
&\ge\varliminf_{t\rightarrow 0}\int_{F}|b(x)-b(y)||\phi_{10\sqrt{n}\delta^{-1}l_Q}(x-y)-\phi_{t}(x-y)|dy
\Big(\int_F\mu(x)dx\Big)^{-1/p}\\
&\geq \int_{F}|b(x)-b(y)|\varliminf_{t\rightarrow 0}|\phi_{10\sqrt{n}\delta^{-1}l_Q}(x-y)-\phi_{t}(x-y)|dy\Big(\int_F\mu(x)dx\Big)^{-1/p}\\
&=\int_{F}|b(x)-b(y)||\phi_{10\sqrt{n}\delta^{-1}l_Q}(x-y)|dy\Big(\int_F\mu(x)dx\Big)^{-1/p}\\
&\gtrsim a_{\tau}(b;Q)
\Big(\int_F\mu(x)dx\Big)^{-1/p},
\end{align*}
which yields that
\begin{align}\label{eq4.4}
\int_E\mathcal{V}_{\rho}(\Phi\ast f)(x)dx\gtrsim\tau|Q|a_{\tau}(b;Q)
\Big(\int_P\mu(x)dx\Big)^{-1/p}.
\end{align}

On the other hand, by the H\"{o}lder inequality and Theorem \ref{th2.1}, we have
\begin{align*}
\int_E\mathcal{V}_{\rho}(\Phi\star f)(x)dx&\leq\Big(\int_E\mathcal{V}_{\rho}(\Phi\star f)(x)^p\lambda(x)dx\Big)^{1/p}\Big(\int_Q\lambda(x)^{-p'/p}dx\Big)^{1/p'}\\
&\lesssim\Big(\int_Q\lambda(x)^{-p'/p}dx\Big)^{1/p'}.
\end{align*}
This, together with \eqref{eq4.4} and $P\subset KQ$ for some $K>0$, gives that
\begin{align}\label{eq4.5}
a_{\tau}(b;Q)&\lesssim\Big(\frac{1}{|Q|}\int_Q\mu(x)dx\Big)^{1/p}
\Big(\frac{1}{|Q|}\int_Q\lambda(x)^{-p'/p}dx\Big)^{1/p'}.
\end{align}
Noting that
$$\frac{1}{|Q|}\int_Q\mu(x)dx\lesssim\Big(\frac{1}{|Q|}\int_Q\mu(x)^{1/(p+1)}dx\Big)^{p+1}$$
(see \cite{LOR1}), using H\"{o}lder's inequality and $\mu=\nu^p\lambda$, we obtain
$$\Big(\frac{1}{|Q|}\int_Q\mu(x)^{1/(p+1)}dx\Big)^{p+1}\leq\Big(\frac{1}{|Q|}\int_Q\nu(x)dx\Big)^p
\Big(\frac{1}{|Q|}\int_Q\lambda(x)dx\Big).$$
Thus, by \eqref{eq4.5} and $\lambda\in A_p$, we conclude that
\begin{align*}
a_{\tau}(b;Q)&\lesssim\Big(\frac{1}{|Q|}\int_Q\nu(x)dx\Big)
\Big(\frac{1}{|Q|}\int_Q\lambda(x)dx\Big)^{1/p}
\Big(\frac{1}{|Q|}\int_Q\lambda(x)^{-p'/p}dx\Big)^{1/p'}
\lesssim\frac{1}{|Q|}\int_Q\nu(x)dx.
\end{align*}
This implies that $b\in BMO_\nu(\Rn)$ by choosing $\tau=1/2^{n+2}$ and invoking \eqref{eq4.1}.

Next, we show that $(2)\Rightarrow (1)$. Indeed, using Lemma \ref{lm2.4}, following the standard steps of \cite{LOR}, there exist $3^n$ sparse families $\mathcal{S}_j$ such that
\begin{equation}\label{eq4.4-1}
\mathcal{V}_\rho((\Phi\star f)_b)(x)\lesssim \sum_{j=1}^{3^n}(\mathcal{T}_{\mathcal{S}_j,b}(f)(x)+\mathcal{T}_{\mathcal{S}_j,b}^\ast
(f)(x)),
\end{equation}
where $$\mathcal{T}_{\mathcal{S},b}f(x):=\sum_{Q\in\mathcal{S}}|b(x)-\langle b\rangle_Q|\langle |f|\rangle_Q\chi_Q(x),\quad\mathcal{T}_{\mathcal{S},b}^\ast f(x):=\sum_{Q\in\mathcal{S}}\langle |(b-\langle b\rangle_Q)f|\rangle_Q\chi_Q(x).$$
In \cite{LOR}, the authors proved that
\begin{equation*}
\|\mathcal{T}_{\mathcal{S},b}f+\mathcal{T}_{\mathcal{S},b}^\ast\|_{L^p(\lambda)}\lesssim
([\mu]_{A_p}[\lambda]_{A_p})^{\max\{1,\frac{1}{p-1}\}}\|b\|_{BMO_\nu(\Rn)}\|f\|_{L^p(\mu)},
\end{equation*}
where $\mu,\lambda\in A_p~(1<p<\infty)$, $\nu=(\mu\lambda^{-1})^{1/p}$ and $b\in BMO_\nu(\Rn)$. This, together with \eqref{eq4.4-1}, shows that $(2)$ implies $(1)$.
 Theorem \ref{new theorem1} is proved.
\end{proof}

\section{The characterization of $\mathcal{BMO}_\omega(\Rn)$ spaces}

This section is devoted to the proof of Theorem \ref{new theorem}. We first recall and establish two lemmas.

\begin{lemma}{\rm(cf. \cite{LKY})}\label{new lemma2}
Let $\tilde{\phi}\in \mathcal{S}(\mathbb{R}^n)$ such that $\tilde{\phi}(x)=1$ for all $x\in B(0,1)$, and $M_{\tilde {\phi}}$ be defined as in
(\ref{eq1.5}). Suppose that $f$ is a measurable function such that $\supp f\subset B:=B(x_B,r)$ with some $x_B\in \mathbb{R}^n$ and $r\in(0,\infty)$. Then for all $x\not\in B$,
\begin{align*}
\frac{1}{|x-x_B|^n}\Big|\int_{B}f(y)dy\Big|\lesssim M_{\tilde{\phi}}(f)(x).
\end{align*}
\end{lemma}

\begin{lemma}\label{new lemma3}
Let $\omega\in A_q$ with $q\in(1,1+1/n)$. Then for any $b\in BMO(\mathbb{R}^n)$ and $(1,q,s)_\omega$-atom $a$ with
$s\ge 0$ and $\supp a\subset B:=B(x_B,r)$, there holds that
\begin{align*}
\|(b-\langle b\rangle_B)\mathcal{V}_\rho(\Phi\star a)\|_{L^1(\omega)}\lesssim\|b\|_{BMO(\mathbb{R}^n)}.
\end{align*}
\end{lemma}
\begin{proof}
We prove this lemma by considering the following two parts:
\begin{align*}
I_1:=\int_{4B}|b(x)-\langle b\rangle_{B}|\mathcal{V}_\rho(\Phi\star a)\omega(x)dx,
\end{align*}
and
\begin{align*}
I_2:=\int_{(4B)^c}|b(x)-\langle b\rangle_{B}|\mathcal{V}_\rho(\Phi\star a)\omega(x)dx.
\end{align*}

Note that for any $\omega\in A_\infty$, $q\in[1,\infty)$ and $B\subset \mathbb R^n$,
\begin{align}\label{eq5.1}
\Big[\frac{1}{\omega(B)}\int_B|b(x)-\langle b\rangle_B|^q\omega(x)dx\Big]^{1/q}\lesssim\|b\|_{BMO(\Rn)}.
\end{align}
By H\"{o}lder's inequality, Theorem \ref{th2.1} and Definition \ref{def3.1}, we have
\begin{align*}
I_1&\leq\Big(\int_{4B}|b(x)-\langle b\rangle_{B}|^{q'}\omega(x)dx\Big)^{1/q'}
\Big(\int_{\mathbb{R}^n}\mathcal{V}_\rho(\Phi\star a)(x)^q\omega(x)dx\Big)^{1/q}\\
&\lesssim\Big[\Big(\int_{4B}|b(x)-\langle b\rangle_{4B}|^{q'}\omega(x)dx\Big)^{1/q'}+
\omega(4B)^{1/q'}\|b\|_{BMO(\mathbb{R}^n)}\Big]\|a\|_{L^q(\omega)}\\
&\lesssim\omega(4B)^{1/q'}\|b\|_{BMO(\mathbb{R}^n)}\omega(B)^{-1/q'}
\sim\|b\|_{BMO(\mathbb{R}^n)}.
\end{align*}

For $I_2$, noting that $\omega\in A_q$ and invoking the vanishing property of $a$, it follows from \eqref{eq3.4} and \eqref{eq5.1} that
\begin{align*}
I_2&\lesssim\int_{(4B)^c}|b(x)-\langle b\rangle_B|\int_B|a(y)|\frac{|y-x_B|}{|x-x_B|^{n+1}}dy\omega(x)dx\\
&\lesssim\int_B|a(y)|\sum_{j=2}^{\infty}\int_{2^{j+1}B\backslash 2^jB}|b(x)-\langle b\rangle_B|\frac{|y-x_B|}{|x-x_B|^{n+1}}\omega(x)dxdy\\
&\leq\Big(\int_B|a(y)|^q\omega(y)dy\Big)^{1/q}\Big(\int_B\omega(y)^{-q'/q}dy\Big)^{1/q'}\\
&\qquad\times\sum_{j=2}^{\infty}\int_{2^{j+1}B}\frac{r}{(2^jr)^{n+1}}(|b(x)-\langle b\rangle_{2^{j+1}B}|+|\langle b\rangle_{2^{j+1}B}-\langle b\rangle_B|)\omega(x)dx\\
&\lesssim\frac{|B|}{\omega(B)}\sum_{j=2}^{\infty}2^{-j(n+1)}j\frac{\omega(2^{j+1}B)}{|B|}
\|b\|_{BMO(\mathbb{R}^n)}
\lesssim\|b\|_{BMO(\mathbb{R}^n)}.
\end{align*}
Combining the estimates of $I_1$ and $I_2$, we finish the proof of Lemma \ref{new lemma3}.
\end{proof}

Now, we are in the position to prove Theorem \ref{new theorem}.
\begin{proof}[Proof of Theorem \ref{new theorem}]
First, we show that $(2)$ implies $(1)$. In view of Lemma \ref{lm3.2}, we only need to prove that for any $(1,\infty,s)$-atom $a$ with
$s\ge 0$ and $\supp a\subset B:=B(x_B,r)$, there holds that
\begin{align*}
\|\mathcal{V}_\rho((\Phi\star a)_b)\|_{L^1(\omega)}\lesssim\|b\|_{\mathcal{BMO}_\omega(\mathbb{R}^n)}.
\end{align*}

Write
\begin{align*}
\|\mathcal{V}_\rho((\Phi\star a)_b)\|_{L^1(\omega)}\leq
\|\mathcal{V}_\rho(\Phi\star((b-\langle b\rangle_B)a))\|_{L^1(\omega)}+
\|(b-\langle b\rangle_B)\mathcal{V}_\rho(\Phi\star a)\|_{L^1(\omega)}.
\end{align*}
Since $(1,\infty,s)$-atom is $(1,q,s)$-atom and $\|b\|_{BMO(\mathbb{R}^n)}\lesssim\|b\|_{\mathcal{BMO}_\omega(\mathbb{R}^n)}$, by Lemma \ref{new lemma3} and Theorem \ref{theorem1.1}, it suffices to show that $(b-\langle b\rangle_B)a\in H^1(\omega)$ with
\begin{align}\label{new equation1}
\|(b-\langle b\rangle_B)a\|_{H^1(\omega)}\lesssim\|b\|_{\mathcal{BMO}_\omega(\mathbb{R}^n)}.
\end{align}

We now show \eqref{new equation1}. For $x\not\in2B$, note that
\begin{align*}
M_\phi((b-\langle b\rangle_B)a)(x)&\leq\sup_{t>0}t^{-n}\int_B|b(y)-\langle b\rangle_B||a(y)|\Big|\phi\Big(\frac{x-y}{t}\Big)\Big|dy\\
&\lesssim\sup_{t>0}t^{-n}(1+|x-y|/t)^{-n}\int_B|b(y)-\langle b\rangle_B||a(y)|dy\\
&\lesssim\frac{1}{|x-x_B|^n}\int_B|b(y)-\langle b\rangle_B||a(y)|dy.
\end{align*}
Hence, by the definition of $\mathcal{BMO}_\omega(\mathbb{R}^n)$ and $\|a\|_{L^\infty(\Rn)}\leq\omega(B)^{-1}$, we have
\begin{align}\label{new equation2}
&\int_{(2B)^c}M_\phi((b-\langle b\rangle_B)a)(x)\omega(x)dx\\
&\qquad\quad\lesssim\frac{1}{\omega(B)}\Big(\int_{(2B)^c}\frac{\omega(x)}{|x-x_B|^n}dx\Big)
\Big(\int_B|b(y)-\langle b\rangle_B|dy\Big)\leq\|b\|_{\mathcal{BMO}_\omega(\mathbb{R}^n)}.\nonumber
\end{align}
Meanwhile, by the $L^q(\omega)$-boundedness of $M_\phi$, $\|a\|_{L^\infty(\Rn)}\leq\omega(B)^{-1}$ and \eqref{eq5.1}, we obtain
\begin{align*}
\int_{2B}M_\phi((b-\langle b\rangle_B)a)(x)\omega(x)dx
&\leq\Big[\int_{\mathbb{R}^n}M_\phi((b-\langle b\rangle_B)a)(x)^q\omega(x)dx\Big]^{1/q}
\Big(\int_{2B}\omega(x)dx\Big)^{1/q'}\\
&\lesssim\|(b-\langle b\rangle_B)a\|_{L^q(\omega)}\omega(2B)^{1/q'}\\
&\lesssim\omega(B)^{-1/q}\Big(\int_B|b(x)-\langle b\rangle_B|^q\omega(x)dx\Big)^{1/q}\\
&\lesssim\|b\|_{BMO(\mathbb{R}^n)}\lesssim\|b\|_{\mathcal{BMO}_\omega(\mathbb{R}^n)}.
\end{align*}
This, together with \eqref{new equation2}, implies \eqref{new equation1}, and proves that $(2) \Rightarrow (1)$.

Next, we show that $(1)\Rightarrow(2)$. For any ball $B:=B(x_B,r)$, take $h:={\rm sgn}(b-\langle b\rangle_B)$ and
\begin{align*}
a:=\frac{1}{2\omega(B)}(h-\langle h\rangle_B)\chi_B.
\end{align*}
Then $\supp a\subset B$, $\|a\|_{L^\infty(\Rn)}\leq\omega(B)^{-1}$ and $\int_Ba(y)dy=0.$ By Lemma \ref{new lemma3} and assumption $(1)$ in Theorem \ref{new theorem}, we obtain
\begin{align*}
\|\mathcal{V}_\rho(\Phi\star((b-\langle b\rangle_B)a))\|_{L^1(\omega)}&\leq\|\mathcal{V}_\rho((\Phi\star a)_b)\|_{L^1(\omega)}+
\|(b-\langle b\rangle_B)\mathcal{V}_\rho(\Phi\star a)\|_{L^1(\omega)}\\
&\lesssim\|a\|_{H^1(\omega)}+\|b\|_{BMO(\mathbb{R}^n)}.
\end{align*}
Hence, by (i) of Theorem \ref{theorem1.1} with $\lim\limits_{t\rightarrow0}\phi_t\ast f=f$ on $L^1(\omega)$ for $\omega\in A_1$ (see \cite{MuW}) and \eqref{eq5.1},
\begin{align*}
\|(b-\langle b\rangle_B)a\|_{H^1(\omega)}
&\leq\|(b-\langle b\rangle_B)a\|_{L^1(\omega)}+\|\mathcal{V}_\rho(\Phi\star((b-\langle b\rangle_B)a))\|_{L^1(\omega)}\\
&\lesssim\frac{1}{\omega(B)}\int_B|b(x)-\langle b\rangle_B|\omega(x)dx+\|a\|_{H^1(\omega)}+\|b\|_{BMO(\mathbb{R}^n)}\\
&\lesssim\|b\|_{BMO(\mathbb{R}^n)}+\|a\|_{H^1(\omega)}+\|b\|_{BMO(\mathbb{R}^n)}\lesssim 1.
\end{align*}
Also, invoking Lemma \ref{new lemma2}, for any $x\not\in B$, we have
\begin{align*}
&\frac{1}{2\omega(B)|x-x_B|^n}\int_B|b(y)-\langle b\rangle_B|dy
\\
&\qquad\quad=\frac{1}{|x-x_B|^n}\int_B(b(y)-\langle b\rangle_B)a(y)dy\lesssim M_{\tilde{\phi}}((b-\langle b\rangle_B)a)(x).
\end{align*}
Consequently,
\begin{align*}
\Big(\frac{1}{\omega(B)}\int_{B^c}\frac{\omega(x)}{|x-x_B|^n}dx\Big)\Big(\int_B|b(y)-\langle b\rangle_B|dy\Big)
&\lesssim\|M_{\tilde{\phi}}((b-\langle b\rangle_B)a)\|_{L^1(\omega)}\\
&\lesssim\|(b-\langle b\rangle_B)a\|_{H^1(\omega)},
\end{align*}
which implies that
\begin{align*}
\|b\|_{\mathcal{BMO}_\omega(\mathbb{R}^n)}\lesssim\sup_B\|(b-\langle b\rangle_B)a\|_{H^1(\omega)}
\lesssim 1.
\end{align*}
This finishes the proof of the implication $(1) \Rightarrow (2)$. Theorem \ref{new theorem} is proved.
\end{proof}

%\subsection*{Acknowledgements} The authors thank the referee cordially for their valuable suggestions.
%%%%%%%%%%%% References %%%%%%%%%%%%%
%%
%<Author name> is written as Initial of Given Name, and Family Name.
%<Title> is written in roman letters.
%<Journal name> should be abbreviated according to
% the MR Serials Abbreviations List of Mathematical Reviews:
% (Abbreviations of Names of Serials; http://www.ams.org/mr-database)
%For <Pages>, use en-dash "--" between page numbers.
%%


\begin{thebibliography}{99}

%%% Journals %%%

%<Author name>, <Title>, <Journal name> <Volume> (<Year>), <Pages>.
\bibitem{BetFHR}J.J. Betancor, J.C. Farina, E. Harbour and L. Rodriguez-Mesa,
$L^p$-boundedness properities of variation operators in the Schr\"odinger setting,
Rev. Mat. Complut. 26(2) (2013), 485--534.

\bibitem{Bl}S. Bloom,
A commutator theorem and weighted BMO,
Trans. Amer. Math. Soc. 292(1) (1985), 103--122.

\bibitem{Bo}J. Bourgain,
Pointwise ergodic theorems for arithmetric sets,
Publ. Math. Inst. Hautes \'{E}tudes Sci. 69(1) (1989), 5--45.

\bibitem{BLYZ}M. Bownik, B. Li, D. Yang and Y. Zhou,
Weighted anisotropic Hardy spaces and their applications in boundedness of sublinear operators,
Indiana Univ. Math. J. 57(7) (2008), 3065--3100.

\bibitem{CaJRW1}J.T. Campbell, R.L. Jones, K. Reinhold and M. Wierdl,
Oscillations and variation for the Hilbert transform,
Duke Math. J. 105(1) (2000), 59--83.

\bibitem{CaJRW2}J.T. Campbell, R.L. Jones, K. Reinhold and M. Wierdl,
Oscillations and variation for singular integrals in higher dimensions,
Trans. Amer. Math. Soc. 355(5) (2003), 2115--2137.

\bibitem{CheDHL}Y. Chen, Y. Ding, G. Hong and H. Liu,
Weighted jump and variational inequalities for rough operators.
J. Funct. Anal. 275(8) (2018), 2446--2475.

\bibitem{CheDHL2}Y. Chen, Y. Ding, G. Hong and H. Liu,
Variational inequalities for the commutators of rough operators with BMO functions,
Sci. China Math. (in press)

\bibitem{CRW}R.R. Coifman, R. Rochberg and G. Weiss,
Factorization theorems for Hardy spaces in several variables,
Ann. of Math. 103 (1976), 611-635.

\bibitem{CrMacMTV}R. Crescimbeni, R.A. Mac\'{i}as, T. Men\'{a}rguez, J.L. Torrea and B. Viviani,
The $\rho$-variation as an operator between maximal operators and singular integrals,
J. Evol. Equ. 9(1) (2009), 81--102.

\bibitem{DHL}Y. Ding, G. Hong and H. Liu,
Jump and variational inequalities for rough operators,
J. Fourier. Anal. Appl. 23(3) (2017), 679--711.

\bibitem{Gar}J. Garc\'ia-Cuerva, Weighted $H^p$ spaces.
Dissertationes Math. (Rozprawy Mat.) 162 (1979), 63 pp.

\bibitem{GT}T.A. Gillespie and J.L. Torrea,
Dimension free estimates for the oscillation of Riesz transform,
Israel J. Math. 141 (2004), 125--144.

\bibitem{HLW}I. Holmes, M.T. Lacey and B.D. Wick,
Commutators in the two-weight setting,
Math. Ann. 367(1-2) (2016), 51--80.

\bibitem{HRS}I. Holmes, R. Rahm and S. Spencer,
Commutators with fractional operatros,
Studia Math. 233(3) (2016), 279--291.

\bibitem{HLP} T.P. Hyt\"onen, M. Lacey and C. P\'erez,
Sharp weighted bounds for the q-variation of singular integrals,
Bull. Lond. Math. Soc. 45(3) (2013), 529--540.

\bibitem{JR}R.L. Jones and K. Reinhold,
Oscillation and variation inequalities for convolution powers,
Ergodic Theory Dynam. Sys. 21(6) (2001), 1809--1829.

\bibitem{Le}D. L\'{e}pingle,
La variation d'order $p$ des semi-martingales,
Z. Wahrsch. Verw. Gebiete 36(4) (1976), 295--316.

\bibitem{Ler}A.K. Lerner,
On pointwise estimates involving sparse operators,
New York J. Math. 22 (2017), 341--349.

\bibitem{LOR}A.K. Lerner, S. Ombrosi and I.P. Rivera-R\'{i}os,
On pointwise and weighted estimates for commutators of Calder\'{o}n-Zygmund operators,
Adv. Math. 319 (2017), 153--181.

\bibitem{LOR1}A.K. Lerner, S. Ombrosi and I.P. Rivera-R\'{i}os,
Commutators of singular integrals revisited,
Bull. Lond. Math. Soc. 51 (2019), 107--119.

\bibitem{LKY}Y. Liang, L.D. Ky and D. Yang,
Weighted endpoint estimates for commutators of Calder\'{o}n-Zygmund operators,
Proc. Amer. Math. Soc. 144(12) (2016), 5171--5181.

\bibitem{LW}F. Liu and H. Wu,
A criterion on oscillation and variation for the commutators of singular integrals,
Forum Math. 27 (2015), 77--97.

\bibitem{Liu}H. Liu,
Variational characterization of $H^p$,
Proc. Roy. Soc. Edinburgh Sect. A 149(5) (2019), 1123--1134.

\bibitem{MaTX1}T. Ma, J.L. Torrea and Q. Xu,
Weighted variation inequalities for differential operators and singular integrals,
J. Funct. Anal. 268(2) (2015), 376--416.

\bibitem{MaTX2}T. Ma, J.L. Torrea and Q. Xu,
Weighted variation inequalities for differential operators and singular integrals in higher dimensions,
Sci. China Math. 60(8) (2017), 1419–-1442.

\bibitem{MuW}B. Muckenhoupt and R.L. Wheeden,
On the dual of weighted $H^{1}$ of the half-space,
Studia Math. 63(1) (1978), 53--79.

\bibitem{NTV} F. Nazarov, S. Treil, S. and A. Volberg,
Weak type estimates and Cotlar inequalities for Calder\'on--Zygmund operators on nonhomogeneous spaces,
Internat. Math. Res. Notices 9 (1998), 463--487.

\bibitem{Pa}M. Paluszy\'{n}ski,
Characterization of the Besov spaces via the commutator operator of Coifman, Rochberg and Weiss,
 Indiana Univ. Math. J. 44(1) (1995), 1--17.

\bibitem{Perey}M.C. Pereyra,
Dyadic harmonic analysis and weighted inequalities: the sparse revolution,
New Trends in Applied Harmonic Analysis, Volume 2. Birkh\"{a}user, Cham (2019): 159--239.

\bibitem{ST}J.-O. Str\"omberg and A. Torchinsky,
Weighted Hardy Spaces,
Lecture Notes in Mathematics, 1381. Springer-Verlag, Berlin, 1989.

\bibitem{WGW}Y. Wen, W. Guo and H. Wu,
Variation and oscillation inequalities for commutattors in two-weight setting,
Forum Math. 2020, https://doi.org/10.1515/forum-2019-0217.

\bibitem{WWZ}Y. Wen, H. Wu and J. Zhang,
Weighted variation inequalities for singular integrals and commutators,
J. Math. Anal. Appl. 485 (2020), 123825, 16pages.

\end{thebibliography}
\end{document}